\documentclass[fleqn, eqnumis, pdf]{hjms}

\usepackage{footnote}
\usepackage{multicol}
\usepackage{booktabs}
\usepackage[flushleft]{threeparttable}
\usepackage[font=small,labelfont=bf]{caption}

\usepackage[T1]{fontenc}
\usepackage{amsmath,amsthm}
\usepackage{amssymb,latexsym}
\usepackage{enumerate}

\theoremstyle{plain}
\newtheorem{f}[theorem]{Fact}
\theoremstyle{definition}
\theoremstyle{remark}

\begin{document}
\title{Bornological  quasi-metrizability in generalized topology}

\author{ Artur Pi\k{e}kosz\footnote{ Institute of Mathematics, Cracow University of Technology, Warszawska 24, 31-155 Cracow, Poland, 
 \ Email: {\tt pupiekos@cyfronet.pl }}\footnote{Corresponding Author.} and 
Eliza Wajch\footnote{ Institute of Mathematics and Physics, Siedlce University of Natural Sciences and Humanities, 3 Maja 54, 08-110 Siedlce, Poland,
  \ Email: {\tt  eliza.wajch@wp.pl }} }
{}

\begin{abstract}
A concept of quasi-metrizability with respect to a bornology of a generalized topological space in the sense of Delfs and Knebusch is introduced.
Quasi-metrization theorems for generalized bornological universes are deduced. A uniform quasi-metrizability with respect to a bornology is studied. The class of locally small spaces is considered and a possibly larger class of weakly locally small spaces is defined. The proofs and numerous examples are given in \textbf{ZF}. An example of a weakly locally small space which is not locally small is constructed under \textbf{ZF+CC}. Several  categories, relevant to generalized bornological universes, are defined and shown to be topological constructs. 
\end{abstract}


\begin{AMS}
{54A05, 18B30 (primary),  54E35, 54E55 (secondary)}
\end{AMS}

\section{Introduction}

 The present work is the first study of problems relevant to quasi-metrizability in the class of Delfs and Knebusch generalized topological spaces (gts\-es),  under the basic 
 set-theoretic assumption of \textbf{ZF} and its consistency. 
 
Contrary to very simple generalizations of topological spaces given, for example, in \cite{AF} and \cite{Cs}, the notion of a Delfs-Knebusch generalized topological space, introduced in \cite{DK}, is more powerful,  since it originates from a categorial concept of Grothendieck topology (cf. \cite{DK}, \cite{EP}, \cite{Pie1} and \cite{Pie3}). Let us admit  that many mathematicians have already studied generalized topological spaces in the sense of Cs\'asz\'ar's article \cite{Cs}. Although our article does not concern  Cs\'asz\'ar's style generalization of topologies, we thank T. Kubiak for turning our attention to the following fact: generalizations of topologies such that it is not assumed that finite intersections of open sets are open had appeared in \cite{AF}, earlier than in~\cite{Cs}. 

Unfortunately, only several articles about Delfs-Knebusch gtses have been published so far (cf. \cite{Pie1}, \cite{Pie2}, \cite{Pie3} and \cite{PW}). However, other articles on topics relevant to \cite{DK} have also appeared. For instance, Edmundo and Prelli, while applying some results  of \cite{DK} in \cite{EP}, have recently used a difficult language of 
$\mathcal{T}$-topologies, locally weakly quasi-compact spaces 
and sheaves on them. Since quasi-metrizability is the main topological problem of our work, 
we prefer to use the language of Delfs-Knebusch 
from \cite{DK}: gtses may be seen as topological spaces 
with some additional structure. The original definition 
of a generalized topological space given in \cite{DK} 
was simplified in \cite{Pie1}. It was shown in 
\cite{Kn}, \cite{Pie2} and \cite{Pie3} 
that in the context of locally definable spaces 
it is sufficient to use  function sheaves 
instead of general sheaves. Our direct approach to the 
investigation of Delfs-Knebusch gtses seems simpler and 
more natural for problems of general topology than that 
of \cite{EP}.

To avoid misunderstandings, let us make it more precise what \textbf{ZF} is in this article. Since part of our results concern proper classes in category theory, while it is disturbingly assumed in \cite{Ku} that proper classes do not exist (cf. pages 14 and 34 of \cite{Ku}), so far as \textbf{ZF} is concerned, in much the same way, as in \cite{PW}, we follow \cite{MacL} and assume the existence of a universe $\mathbb{U}$ (cf. pages 22 and 23 of \cite{MacL}). Sets in the sense of \cite{MacL} are called \emph{totalities} or \emph{collections}. A \emph{class} is a collection $u\subseteq\mathbb{U}$. Elements $u\in\mathbb{U}$ are called $\mathbb{U}$-small sets in \cite{MacL}. We denote by \textbf{ZF} the system of axioms which consists of the existence of a universe $\mathbb{U}$ and axioms 0-8 from pages 9-10 of \cite{Ku} for sets in the sense of \cite{MacL}. This system does not contain the axiom of choice. From now on, a totality $u$ will be called a \emph{set} if and only if $u$ is a $\mathbb{U}$-small set. A \emph{proper class} is a class $u$ such that $u\notin \mathbb{U}$. Our notation concerning other set-theoretic axioms independent of \textbf{ZF} that are used in this article is the same as in \cite{Her}. 
We clearly denote the results that are obtained not in $\mathbf{ZF}$ but under  $\mathbf{ZF+CC}$ where $\mathbf{CC}$ is the axiom of countable choice (cf. Definition 2.5 of \cite{Her}). 
Of course, not all axioms of \textbf{ZF} are needed to deduce some results. For instance, our results in \textbf{ZF} that do not involve proper classes, can be deduced from the standard system of axioms 0-8 given on pages 9-10 of \cite{Ku}. 

The following definition, equivalent to the notion of a gts in the sense of Delfs and Knebusch, is a reformulation of Definition 2.2.2 from \cite{Pie1}:
 
 \begin{definition}
 A \emph{generalized topological space} in the sense of Delfs and Knebusch (abbreviation: gts) is a triple $(X, \text{Op}_X, \text{Cov}_X)$ where $X$ is a set for which $\text{Op}_X\subseteq \mathcal{P}(X)$, while $\text{Cov}_X\subseteq \mathcal{P}(\text{Op}_X)$ and the following conditions are satisfied:
\begin{enumerate}
\item[(i)]  if $\mathcal{U}\subseteq \text{Op}_X$ and $\mathcal{U}$ is finite, then $\bigcup\mathcal{U}\in\text{Op}_X, \bigcap\mathcal{U}\in\text{Op}_X$ and  $\mathcal{U}\in \text{Cov}_X$ (where $\bigcap \varnothing =X$);
\item[(ii)] if $\mathcal{U}\in\text{Cov}_X$ and $ V\in\text{Op}_X$, then 
$\{ U\cap V: U\in\mathcal{U}\}\in\text{Cov}_X$;
\item[(iii)] if $\mathcal{U}\in\text{Cov}_X$ and, for each $U\in \mathcal{U}$, we have $\mathcal{V}(U)\in\text{Cov}_X$ such that $\bigcup\mathcal{V}(U)=U$, then $\bigcup_{U\in\mathcal{U}}\mathcal{V}(U)\in\text{Cov}_X$;
\item[(iv)] if $\mathcal{U}\subseteq \text{Op}_X$ and $\mathcal{V}\in \text{Cov}_X$ are such that $\bigcup\mathcal{V}=\bigcup\mathcal{U}$ and, for each $V\in\mathcal{V}$ there exists $U\in\mathcal{U}$ such that $V\subseteq U$, then $\mathcal{U}\in\text{Cov}_X$;
\item[(v)] if $\mathcal{U}\in\text{Cov}_X$, $V\subseteq\bigcup_{U\in\mathcal{U}}U$ and, for each $U\in\mathcal{U}$, we have $V\cap U\in\text{Op}_X$, then $V\in\text{Op}_X$.
\end{enumerate}
\end{definition}

\begin{remark} If $(X, \text{Op}_X, \text{Cov}_X)$ is a gts, 
then $\text{Op}_X=\bigcup\text{Cov}_X$ and, therefore, 
we can identify the gts with the ordered pair 
$(X, \text{Cov}_X)$ (cf. \cite{Pie1}, \cite{PW}). 
If this is not misleading, we shall denote a gts 
$(X, \text{Cov}_X)$ by~$X$.
\end{remark}

In our approach to the problem of how to define a quasi-metrizable gts, we apply bornologies. According to \cite{Hu}, a \emph{bornology} in a set $X$ is a non-empty ideal $\mathcal{B}$ of subsets of $X$ such that each singleton of $X$ is a member of $\mathcal{B}$. A \emph{base of a bornology} $\mathcal{B}$ is a collection $\mathcal{B}_0\subseteq\mathcal{B}$ such that each member of $\mathcal{B}$ is a subset of a member of $\mathcal{B}_0$. A bornology is \emph{second-countable} if it has a countable base. A \emph{bornological universe} is an ordered pair $((X, \tau), \mathcal{B})$ where $(X, \tau)$ is a topological space and $\mathcal{B}$ is a bornology in $X$ (cf. Definition 1.2 of \cite{Hu}).

\begin{definition} A \emph{generalized bornological universe} is an ordered pair $(X, \mathcal{B})$ where $X=(X, \text{Cov}_X)$ is a gts and $\mathcal{B}$ is a bornology in the set $X$.
\end{definition}

 A \emph{quasi-pseudometric} on a set $X$ is a function $d: X\times X\to [0; +\infty)$ such that, for all $x,y,z\in X$, $d(x,y)\leq d(x,z)+d(z,y)$ and $d(x,x)=0$. 
 A quasi-pseudometric $d$ on $X$ is called a \emph{quasi-metric} if, for all $x,y\in X$, the condition $d(x,y)=0$ implies $x=y$ (cf. \cite{Kel}, \cite{FL}). 
 Let $d$ be a quasi-pseudometric on $X$. The \emph{conjugate} of $d$ is the quasi-pseudometric $d^{-1}$ defined by $d^{-1}(x, y)=d(y, x)$ for $x, y\in X$. 
 The $d$-\emph{ball with centre $x\in X$ and radius $r\in(0; +\infty)$} is the set $B_{d}(x, r)=\{ y\in X: d(x, y)<r\}$. 
 For a set $A\subseteq X$ and a number $\delta\in (0; +\infty)$, the \emph{$\delta$-neighbourhood of $A$ with respect to $d$} is  the set $[A]^{\delta}_d=\bigcup_{a\in A}B_d(a,\delta)$. 
 The collection $\tau(d)=\{ V\subseteq X: \forall_{x\in V}\exists_{n\in\omega} B_{d}(x, \frac{1}{2^n})\subseteq V\}$ is the \emph{topology in $X$ induced by $d$}. 
 The triple $(X,\tau(d), \tau(d^{-1}))$ is the \emph{bitopological space associated with $d$}. 

\begin{definition} [cf. Definition 1.5 of \cite{PW2}] Let $d$ be a quasi-(pseudo)metric on a non-empty set $X$ and let $A$ be a subset of $X$. Then:
\begin{enumerate}
\item[(i)]  $A$ is  called \emph{$d$-bounded} if there exist $x\in X$ 
and $r\in (0; +\infty)$ such that $A\subseteq B_d(x, r)$; 
\item[(ii)] if $A$ is not $d$-bounded, we say that $A$ is \emph{$d$-unbounded};
\item[(iii)] $\mathcal{B}(d,X)$ is the collection of all $d$-bounded subsets of $X$.
\end{enumerate}
In addition, if  $X=\emptyset$, one can treat $d=\emptyset$ 
as the unique quasi-pseudometric on $X$ and, of course, 
the empty set should be also called $d$-bounded in the case of the empty space. 
\end{definition}

It was shown in Example 1.6 of \cite{PW2} that, for a quasi-metric $d$ on $X$, 
a set $A\subseteq X$ can be both $d$-bounded and $d^{-1}$-unbounded. 

\begin{definition} [cf. Definition 1.7 of \cite{PW2}] A bornological universe $((X, \tau),\mathcal{B})$ is called \emph{quasi-(pseudo)metrizable} if there exists a quasi-(pseudo)metric $d$ on $X$ such that  $\tau=\tau(d)$ and, moreover, $\mathcal{B}$ is the collection of all $d$-bounded sets.
\end{definition}

For a collection $\mathcal{A}$ of subsets of a set $X$, we denote by $\tau(\mathcal{A})$ the weakest among all topologies in $X$ that contain $\mathcal{A}$. For a gts $(X, \text{Op}_X, \text{Cov}_X)$, we call the topological space $X_{top}=(X, \tau(\text{Op}_X))$ the \emph{topologization of the gts} $X$ (cf. \cite{PW}).

\begin{definition} Suppose that $(X, \mathcal{B})$ is a generalized bornological universe. Then we say  that the gts $X$ is \emph{$\mathcal{B}$-(quasi)-(pseudo)\-metrizable} or \emph{(quasi)-(pseudo)\-metrizable with respect to $\mathcal{B}$} if the bornological universe $(X_{top}, \mathcal{B})$ is (quasi)-(pseudo)metri\-zable. 
\end{definition}

\begin{definition}[cf. Definition 1.4 of \cite{PW2}] A \emph{bornological biuniverse} is an ordered pair $((X, \tau_1, \tau_2), \mathcal{B})$  where $(X, \tau_1, \tau_2)$ is a bitopological space and $\mathcal{B}$ is a bornology in~$X$.
\end{definition}

To a great extent, the present work is a continuation of \cite{PW2}. Therefore, let us use the terminology of \cite{PW2}. 

\begin{definition} [cf. Definition 1.3 of \cite{PW2}]  Let $(X, \tau_{1}, \tau_{2})$ be a bitopological space. A bornology $\mathcal{B}$ in $X$ is called $(\tau_{1}, \tau_{2})$\emph{-proper} if, for each $A\in\mathcal{B}$, there exists $B\in\mathcal{B}$ such that $\text{cl}_{\tau_2}A\subseteq \text{int}_{\tau_1}(B)$. If $\tau=\tau_1=\tau_2$ and the bornology  $\mathcal{B}$ is $(\tau,\tau)$-proper, we say that $\mathcal{B}$ is $\tau$\emph{-proper}.
\end{definition}

Let us formulate our first (quasi)-(pseudo)metrization theorem in the class of gtses which follows from the results of \cite{PW2}. The notion of a $(\tau_1, \tau_2)$-characteristic function of a bornology is introduced in Definition 4.4 of \cite{PW2}. 

\begin{theorem}[cf.  Theorems 4.7 and 4.15, Corollaries 4.10 and 4.16 of \cite{PW2}]
 For every gts $(X, \text{Op}_X, \text{Cov}_{X})$
and a bornology $\mathcal{B}$ in $X$, the following conditions are equivalent:
\begin{enumerate}
\item[(i)] $((X, \text{Cov}_X), \mathcal{B})$ is a (quasi)-(pseudo)metrizable generalized bornological universe;
\item[(ii)] there exists a topology $\tau_2$ on $X$ such that 
$(X, \tau(\text{Op}_X), \tau_2)$ is (quasi)-(pseudo)\-metriz\-able and $\mathcal{B}$ is a $(\tau(\text{Op}_X), \tau_2)$-proper bornology with a countable base;
\item[(iii)] 
 there exists a topology $\tau_2$ on $X$ such that 
$(X, \tau(\text{Op}_X), \tau_2)$ is (quasi)-(pseudo)\-metriz\-able and $\mathcal{B}$ admits a $(\tau(\text{Op}_X), \tau_2)$-characteristic function;
\item[(iv)] there exists a (quasi)-(pseudo)metric $d$ on $X$ such that $\tau(\text{Op}_X)=\tau(d)$ and, simultaneously,  $\mathcal{B}$  has a base $\{ B_n: n\in\omega\}$ with the following property:
$$\forall_{n\in\omega}\exists_{\delta\in(0; +\infty)}[B_n]^{\delta}_d\subseteq B_{n+1}.$$
\end{enumerate}
\end{theorem}

In Section 2, we discuss natural bornologies in every gts: the small sets, the relatively compact sets and the relatively admissibly compact sets.
Other (quasi)-(pseudo)\-metrization theorems for gtses are given in Section 3.
The main theorem of Section 3 (Theorem 3.5) gives a necessary and sufficient condition for a gts with its locally small partial topologization to be (quasi)-(pseudo)metrizable with respect to the bornology of small sets. Moreover, a notion of a weakly locally small gts is introduced  and a non-trivial example of a weakly locally small but not locally small gts is constructed in every model for $\mathbf{ZF+CC}$ in Section 3. Applications of the  (quasi)-(pseudo)metrization theorems to numerous examples are shown in Section 4. Finally, in Section 5, we define several new categories relevant to this work and we prove that the newly defined categories are all topological constructs. 

So far as gtses are concerned, we use the terminology of \cite{DK}, \cite{Pie1}, \cite{Pie2} and \cite{PW}.

\begin{definition}[cf. \cite{Pie1}] If $ X=(X, \text{Cov}_X)$ and $Y=(Y, \text{Cov}_Y)$ are gtses, then:
\begin{enumerate}
\item[(i)] a set $U\subseteq X$ is called \emph{open} in the gts $X$ if $U\in\text{Op}_X$;
\item[(ii)] the collection $\text{Cov}_X$ is the \emph{generalized topology} in $X$;
\item[(iii)] an \emph{admissible open family} in the gts $X$ is a member of $\text{Cov}_X$;
\item[(iv)] a mapping $f:Y\to X$ is $(\text{Cov}_Y, \text{Cov}_X)$-\emph{strictly continuous} (in abbreviation: strictly continuous) if, for each $\mathcal{U}\in \text{Cov}_X$, we have $\{ f^{-1}(U):U\in \mathcal{U}\}\in \text{Cov}_Y$.
\end{enumerate}
\end{definition}
 
We denote by $\mathbf{GTS}$ the category of gtses as objects and strictly continuous mappings as morphisms. The category $\mathbf{GTS}$ is a subcategory of the category of Grothen\-dieck sites and their morphisms (cf. \cite{Pie1}, p. 223). Several other categories, relevant to $\mathbf{GTS}$ and bornologies, are defined in Section 5.

While we know that $\mathbf{GTS}$ is a topological construct (cf. Theorem 2.2.60 of \cite{Pie1} and Theorem 4.4 of \cite{PW}), the following problem is open:
\begin{problem} Is $\mathbf{GTS}$ isomorphic to a category of type $(\mathbf{T},\mathcal{V})$-$\textrm{Cat}$ for some quantale $\mathcal{V} $ and monad $\mathbf{T}$ (cf. Definitions III.1.6.1 and III.1.6.3 in \cite{HSW})?
\end{problem}

\section{Fundamental bornologies in gtses}

In this section, we consider natural bornologies in gtses.   

\begin{definition}[cf. Definitions 2.2.13 and 2.2.25 of \cite{Pie1}]
 If $K$ is a subset of a set $X$, then we say that a family $\mathcal{U}\subseteq\mathcal{P}(X)$ is \emph{essentially finite on} $K$ if there exists a finite $\mathcal{V}\subseteq\mathcal{U}$ such that $K\cap\bigcup\mathcal{U}\subseteq\bigcup\mathcal{V}$.
\end{definition}

\begin{definition}[cf. Definition 2.2.25 of \cite{Pie1}]
 If $X=(X,\text{Cov}_X)$ is a gts, then a set $K\subseteq X$ is called \emph{small in the gts $X$} if each family $\mathcal{U}\in\text{Cov}_X$ is essentially finite on~$K$.
\end{definition}

The collection of all small sets of a gts $X$ is a bornology in $X$ (cf. Fact 2.2.30 of \cite{Pie1}). 

\begin{definition} 
For a gts $X$, the \emph{small bornology} of $X$ is the collection 
$\mathbf{Sm}(X)$ of all small sets in~$X$.
\end{definition}

$\mathbf{Sm}(X)$ was denoted by $\text{Sm}_X$ in \cite{Pie1} but, since we use the notation of \cite{PW2}, we have replaced $\text{Sm}_X$ by $\mathbf{Sm}(X)$. The bornology of all finite subsets of $X$ (the smallest bornology of $X$) is denoted by $\mathbf{FB}(X)$.

\begin{definition}[cf. Definition 3.2 of \cite{PW}]
 If $X$ is a gts, we call a set $A\subseteq X$ \emph{admissibly compact} in $X$ if, for each $\mathcal{U}\in\text{Cov}_X$ such that $A\subseteq\bigcup\mathcal{U}$, there exists a finite $\mathcal{V}\subseteq \mathcal{U}$ such that $A\subseteq\bigcup\mathcal{V}$ .
\end{definition}

\begin{definition} For a gts $X$, the \emph{admissibly compact bornology} of $X$ is the collection $\mathbf{ACB}(X)$ of all subsets of admissibly compact sets of the gts $X$. 
\end{definition}

\begin{definition} Let $X$ be a gts. We say that a set $A$ is \emph{topologically compact} in $X$ if $A$ is compact in $X_{top}$ (cf.  Definition 3.2 of \cite{PW}).
The compact bornology $\mathbf{CB}(X_{top})$  (cf. \cite{GM} and \cite{PW2}) will be called the \emph{compact bornology of the gts} $X$ and it will be denoted by $\mathbf{CB}(X)$.
\end{definition}

\begin{f} For every gts $X$, the inclusion $(\mathbf{Sm}(X)\cup\mathbf{CB}(X))\subseteq \mathbf{ACB}(X)$ holds.
\end{f}

The following example shows that it can happen that  $\mathbf{Sm}(X)\cup \mathbf{CB}(X)\neq \mathbf{ACB}(X)$ and neither $\mathbf{Sm}(X)\subseteq\mathbf{CB}(X)$ nor $\mathbf{CB}(X)\subseteq\mathbf{Sm}(X)$.

\begin{example} \label{comp-born}
For $X=\mathbb{R}\times\{0, 1\}$, let $\text{Op}_X$ be the natural topology in $X$ inherited from the usual topology of $\mathbb{R}$ and let $\text{Cov}_X$ be the collection of all families $\mathcal{U}\subseteq \text{Op}_X$ such that $\mathcal{U}$ is essentially finite on $\mathbb{R}\times\{0\}$. Then, for  $A=[0; 1]\times\{1\}$ and $B=\mathbb{R}\times\{0\}$, we have $A\in\mathbf{CB}(X)\setminus\mathbf{Sm}(X)$  and $B\in\mathbf{Sm}(X)\setminus\mathbf{CB}(X)$, while $A\cup B\in\mathbf{ACB}(X)\setminus (\mathbf{CB}(X)\cup\mathbf{Sm}(X))$ . 
\end{example}

For a set $X$ and a collection $\Psi\subseteq\mathcal{P}^2(X)$, we denote by $\langle \Psi\rangle_X$ the smallest among  generalized topologies in $X$ that contain $\Psi$. If $\mathcal{A}\subseteq \mathcal{P}(X)$, let $\text{EssCount}(\mathcal{A})$ be the collection of all essentially countable subfamilies of $\mathcal{A}$. We recall that $\text{EssFin}(\mathcal{A})$ is the collection of all essentially finite subfamilies of $\mathcal{A}$ (cf. \cite{Pie1}-\cite{Pie2} and \cite{PW}).

\begin{f}[cf. Examples 2.2.35 and 2.2.14(8) of \cite{Pie1}]
 Assume $(X, \tau)$ is a topological space. 
 That $\text{EssFin}(\tau)$ is a generalized topology 
 in $X$ is true in $\mathbf{ZF}$. On the other hand, 
 that $\text{EssCount}(\tau)$ is a generalized topology 
 in $X$ is true in $\mathbf{ZF+CC}$.
\end{f}

\begin{remark}
It is unprovable in $\mathbf{ZF}$ that, for every topological space $(X, \tau)$,  the collection $\text{EssCount}(\tau)$ is a generalized topology in $X$. Namely, let $\mathbf{M}$ be a model for $\mathbf{ZF+\neg CC}(\text{fin})$ where $\mathbf{CC}(\text{fin})$ states that countable products on non-empty finite sets are non-empty (cf. Definition 2.9(3) of \cite{Her}). In view of Proposition 3.5 of \cite{Her}, there exists in $\mathbf{M}$ an uncountable set $X$ such that $X$ is a countable union of finite sets. Let $\tau=\mathcal{P}(X)$. If $\text{EssCount}(\tau)$ were a generalized topology in $X$, the family of all singletons of $X$ would belong to $\text{EssCount}(\tau)$ which is impossible, since $X$ is uncountable. 
\end{remark}

Let us observe that, for the gts $X$ from Example \ref{comp-born}, the admissibly compact bornology of $X$ is generated by $\mathbf{CB}(X)\cup\mathbf{Sm}(X)$.
That not every gts may share this property is shown by the following example:

\begin{example}[\textbf{ZF+CC}] For $X=\omega_1$, let $\text{Op}_X$ be the topology induced by the usual linear order in $\omega_1$ and let $\text{Cov}_X=\text{EssCount}(\text{Op}_X)$. Then $\mathbf{Sm}(X)=\mathbf{FB}(X)\neq \mathbf{CB}(X)\neq \mathbf{ACB}(X)=\mathcal{P}(X)$. 
\end{example} 

In what follows, for sets  $X, Y$ with $Y\subseteq X$ and for $\Psi\subseteq \mathcal{P}^2(X)$, we use the notation $\Psi\cap_2Y$  from \cite{Pie1} for the collection of all families $\mathcal{U}\cap_1 Y=\{U\cap Y: U\in\mathcal{U}\}$ where $\mathcal{U}\in\Psi$. We want to describe $\langle\Psi\cap_2 Y\rangle_Y$ more precisely in the case when $\Psi\cap_2 Y\subseteq \text{EssFin}(\mathcal{P}(Y))$. To do this, we need the concept of a full ring of sets in $Y$ that was of frequent use in \cite{PW}. Namely, a \emph{full ring in $Y$} is a collection $\mathcal{C}\subseteq \mathcal{P}(Y)$ such that $\emptyset, Y\in\mathcal{C}$, while $\mathcal{C}$ is closed under finite unions and under finite intersections. For $\mathcal{A}\subseteq \mathcal{P}(Y)$, let $L_Y[\mathcal{A}]$ be the intersection of all full rings in $Y$ that contain $\mathcal{A}$. 

\begin{proposition} For a set $X$, let $\Psi\subseteq \mathcal{P}^2(X)$. Suppose that $Y\subseteq X$ and that each family from $\Psi$ is essentially finite on $Y$. Then  the following conditions are satisfied: 
\begin{enumerate}
\item[(i)] $\langle \Psi\cap_2 Y\rangle_Y=\text{EssFin}( L_Y[\bigcup(\Psi\cap_2 Y)])=\\\text{EssFin}(\bigcup\langle\Psi\cap_2 Y\rangle_Y)=\text{EssFin}(\bigcup\langle \Psi\rangle_X)\cap_2 Y$;
\item[(ii)] each family from $\langle\Psi\rangle_X$ is essentially finite on $Y$.
\end{enumerate}
\end{proposition}
\begin{proof} By applying Proposition 2.2.37 of \cite{Pie1} to the mapping $\text{id}_Y: Y\to X$, we obtain the inclusion $\langle \Psi\rangle_X\cap_2 Y\subseteq \langle\Psi\cap_2 Y\rangle_Y$ which, together with $(i)$, implies $(ii)$. To prove $(i)$, let us put $\mathcal{G}_0=\langle\Psi\cap_2 Y\rangle_Y, \mathcal{G}_1=\text{EssFin}( L_Y[\bigcup(\Psi\cap_2 Y)]), \mathcal{G}_2=\text{EssFin}(\bigcup\mathcal{G}_0)$ and $\mathcal{G}_3=\text{EssFin}(\bigcup\langle \Psi\rangle_X)\cap_2 Y$. Obviously, $\mathcal{G}_0, \mathcal{G}_1$ and $\mathcal{G}_2$ are generalized topologies in $Y$. By Proposition 2.2.53 of \cite{Pie1}, the collection $\mathcal{G}_3$ is also a generalized topology in $Y$. Since $\Psi\cap_2 Y\subseteq \mathcal{G}_1$ and $L_Y[\bigcup(\Psi\cap_2 Y)]\subseteq \bigcup\mathcal{G}_0$, we have $\mathcal{G}_0\subseteq \mathcal{G}_1\subseteq\mathcal{G}_2\subseteq\mathcal{G}_0$. It follows from the inclusion $\langle \Psi\rangle_X\cap_2 Y\subseteq \mathcal{G}_0$ that $\mathcal{G}_3\subseteq \mathcal{G}_0$. 
Since $\bigcup (\langle\Psi\rangle_X \cap_2 Y)$   is a full ring of subsets of $Y$, we get  $\mathcal{G}_1\subseteq \mathcal{G}_3$. This completes our proof to $(i)$.  
\end{proof}

\begin{definition} If $X=(X, \text{Op}, \text{Cov})$ is a gts, then:
\begin{enumerate}
\item[(i)] the \emph{partial topologization} of $X=(X, \text{Op}, \text{Cov})$  is the gts 
$$X_{pt}=(X, (\text{Op})_{pt}, (\text{Cov})_{pt})$$
 where $(\text{Op})_{pt}=\tau(\text{Op})$ and $(\text{Cov})_{pt}=\langle \text{Cov}\cup\text{EssFin}(\tau(\text{Op}))\rangle_X$ (cf. Definition 4.1 of \cite{PW});
\item[(ii)] the gts $X$ is called \emph{partially topological} if $X=X_{pt}$ (cf. Definition 2.2.4 of \cite{Pie1});
\item[(iii)] $\mathbf{GTS}_{pt}$ is the category of all partially topological spaces and strictly continuous mappings, while the mapping $pt:\mathbf{GTS}\to\mathbf{GTS}_{pt}$ is the \emph{functor of partial topologization} defined by: $pt(X)=X_{pt}$ for every gts $X$ and $pt(f)=f$ for every morphism in $\mathbf{GTS}$ (cf. \cite{AHS}, \cite{MacL},  \cite{Pie1} and Definition 4.2 of \cite{PW}).
\end{enumerate}
\end{definition}

\begin{proposition}\label{bornologie-pt}
Let $X$ be a gts. Then $\mathbf{Sm}(X)=\mathbf{Sm}(X_{pt})$, $\mathbf{CB}(X)=\mathbf{CB}(X_{pt})$ and $\mathbf{ACB}(X_{pt})\subseteq\mathbf{ACB}(X)$.
\end{proposition}
\begin{proof}
The equality $\mathbf{CB}(X)=\mathbf{CB}(X_{pt})$ and both the inclusions $\mathbf{Sm}(X_{pt})\subseteq\mathbf{Sm}(X)$ and $\mathbf{ACB}(X_{pt})\subseteq\mathbf{ACB}(X)$ are trivial. Let $X=(X, \text{Op}_X, \text{Cov}_X)$ and let $\Psi=\text{Cov}_X\cup\text{EssFin}(\tau(\text{Op}_X))$. Suppose that $Y\in\mathbf{Sm}(X)$. Since each family from $\Psi$ is essentially finite on $Y$, we infer from Proposition 2.12 that $Y\in\mathbf{Sm}(X_{pt})$.
\end{proof}

\begin{definition}[cf. Proposition 2.2.71 of \cite{Pie1}] Let $\mathcal{L}$ be a full ring of subsets of a set $X$. Then:
\begin{enumerate}
\item[(i)] for a collection $\mathcal{B}\subseteq \mathcal{P}(X)$, we define
$$\text{EF}(\mathcal{L}, \mathcal{B})=\{\mathcal{U}\subseteq\mathcal{L}: \forall 
_{A\in\mathcal{B}} \{A\cap U: U\in\mathcal{U}\}\in\text{EssFin}(\mathcal{P}(A))\};$$
\item[(ii)] for a topology $\tau$ in $X$ and for a bornology $\mathcal{B}$  in $X$, 
the \emph{gts induced by the bornological universe} $((X, \tau), \mathcal{B})$ is $\text{gts}((X, \tau),\mathcal{B})=(X, \tau,\text{EF}(\tau, \mathcal{B}))$.
\end{enumerate}
\end{definition}

In the light of the proof to Proposition 2.1.31 in \cite{Pie2}, we have the following fact:

\begin{f} Suppose that $((X, \tau), \mathcal{B})$ is a bornological universe such that $\tau\cap\mathcal{B}$ is a base for $\mathcal{B}$. Then $\mathbf{Sm}((X, \tau, \text{EF}(\tau, \mathcal{B})))=\mathcal{B}$.
\end{f}

\begin{definition}[cf. Example 2.1.12 of \cite{Pie2}]
 For a (quasi)-(pseudo)metric $d$ on a set $X$, the triple $(X, \tau(d), \text{EF}(\tau(d), \mathcal{B}(d,X)))$ will be called the \emph{gts induced by the (quasi)-(pseudo)metric} $d$.
\end{definition}
 
\begin{f}[cf. Example 2.1.12 of \cite{Pie2}]
 If $d$ is a quasi-(pseudo)metric on a set $X$,  then $ \text{EF}(\tau(d), \mathcal{B}(d,X))$ is a generalized topology in $X$ and  
 $$\mathbf{Sm}((X, \text{EF}(\tau(d), \mathcal{B}(d,X)) )=\mathcal{B}(d,X).$$
\end{f}

\section{$\mathcal{B}$-(quasi)-(pseudo)metrization of gtses}

\begin{definition} 
Let $X$ be a gts and let $\mathcal{S}$ be either $\mathbf{CB}$ or $\mathbf{ACB}$, or $\mathbf{Sm}$. Then we say that $X$ is $\mathcal{S}$-(quasi)-(pseudo)metrizable if $X$ is (quasi)-(pseudo)metrizable with respect to $\mathcal{S}(X)$.
\end{definition}

With Proposition \ref{bornologie-pt} in hand, we can immediately deduce that the following proposition holds:

\begin{proposition} Let $\mathcal{S}$ be either $\mathbf{CB}$ or $\mathbf{Sm}$. Then the following are equivalent for a gts $X$:
\begin{enumerate}
\item[(i)] $X$ is $\mathcal{S}$-(quasi)-(pseudo)metrizable,
\item[(ii)]  $X_{pt}$ is $\mathcal{S}$-(quasi)-(pseudo)metrizable. 
\end{enumerate}
\end{proposition}

\begin{remark} If $X$ is a gts, then the $\mathbf{ACB}$-(quasi)-(pseudo)metrizability of $X_{pt}$ is the (quasi)-(pseudo)metrizability of $X_{pt}$ with respect to $\mathbf{ACB}(X_{pt})$, while the $\mathbf{ACB}$-(quasi)-(pseudo)metrizability of $X$ is equivalent to the (quasi)-(pseudo)metri\-za\-bi\-li\-ty of $X_{pt}$ with respect to $\mathbf{ACB}(X)$. We do not know whether the $\mathbf{ACB}$-(quasi)-(pseudo)\-met\-ri\-zability of $X$ is equivalent to the $\mathbf{ACB}$-(quasi)-(pseudo)metri\-za\-bi\-li\-ty of $X_{pt}$.
\end{remark} 

\begin{definition} 
A gts $X=(X, \text{Op}_X, \text{Cov}_X)$ is called:
\begin{enumerate}
\item[(i)] \emph{locally small} if there exists $\mathcal{U}\in\text{Cov}_X$ such that $\mathcal{U}\subseteq \mathbf{Sm}(X)$ and $X=\bigcup\mathcal{U}$ (cf. Definition 2.1.1 of \cite{Pie2});
\item[(ii)] \emph{weakly locally small} if there exists a collection $\mathcal{U}\subseteq \text{Op}_X\cap\mathbf{Sm}(X)$ such that $X=\bigcup\mathcal{U}$.
\end{enumerate}
\end{definition}

Our next theorem says about the form of the partial topologization of an $\mathbf{Sm}$-(quasi)-(pseudo)metrizable gts $X$ when $X_{pt}$ is locally small. 

\begin{theorem} Suppose that $X=(X, \text{Op}, \text{Cov})$ is a gts such that its partial topologization $X_{pt}=(X,\text{Op}_{pt}, \text{Cov}_{pt})$ is locally small. Then the following conditions are equivalent:
\begin{enumerate}
\item[(i)]$X$ is $\mathbf{Sm}$-(quasi)-(pseudo)metrizable;
\item[(ii)] $X_{pt}$ is induced by some (quasi)-(pseudo)metric $d$.
\end{enumerate}
\end{theorem}
\begin{proof} In view of Proposition \ref{bornologie-pt}, we have $\mathbf{Sm}(X)=\mathbf{Sm}(X_{pt})$. In consequence, it it is obvious that if $X_{pt}$ is induced by a (quasi)-(pseudo)metric $d$, then $X$ is $\mathbf{Sm}$-(quasi)-(pseudo)metrizable.  Assume that $X$ is $\mathbf{Sm}$-(quasi)-(pseudo)metrizable and that $d$ is a (quasi)-(pseudo)metric on $X$ such that $\tau(\text{Op})=\tau(d)$ and $\mathbf{Sm}(X_{pt})$ is the collection of all $d$-bounded sets. Since $X_{pt}$ is locally small, it follows from Proposition 2.1.18 of \cite{Pie2} that $X_{pt}$ is induced by $d$. 
\end{proof} 

\begin{f} If a gts $X$ is induced by a (quasi)-(pseudo)metric, then $X$ is locally small and partially topological.
\end{f}

\begin{f} \begin{enumerate} \item[(i)]If $X$ is a locally small gts, then $X_{pt}$ is locally small.
\item[(ii)] If a gts $X$ is such that $X_{pt}$ is locally small, then $X$ is weakly locally small.
\item[(iii)]  A gts $X$ is weakly locally small if and only if $X_{pt}$ is weakly locally small.
\end{enumerate}
\end{f}

In every model for $\mathbf{ZF+CC}$, we are going to present a construction of an example of a weakly locally small gts $X$ such that $X_{pt}$ is not locally small. 
For $\Psi\subseteq \mathcal{P}^{2}(X)$,  we put $\Psi_0=\Psi$ and, for $n\in\omega$, assuming that the collection $\Psi_{n}\subseteq \mathcal{P}^2(X)$ has been defined, we put $\Psi_{n+1}=(\Psi_n)^+$ where $^+$ is the operator described in the proof of Proposition 2.2.37 in \cite{Pie1}. Then $\langle\Psi\rangle_X=\bigcup_{n\in\omega}\Psi_n$. The symbols $\cup_1, \cap_1, \cup_2, \cap_2$ have the same meaning as in \cite{Pie1}. We shall use the terminology from Definition 2.2.2 of \cite{Pie1}.

\begin{example}$[\mathbf{ZF+CC}]$. Suppose that  $Y$ is an uncountable set. For $n\in\omega$, we put $Y_n=Y\times\{n\}$. Let $X=\bigcup_{n\in\omega} Y_n$, 
$\text{Op}_X=\{ A\subseteq X: \mbox{ for each }n\in \omega \quad
A\cap Y_n \in \mathbf{FB}(Y_n)\} \cup\{ X\}$ 
and $\text{Cov}_X=\text{EF}(\text{Op}_X, \{ Y_n: n\in\omega\})$. The gts $X=(X, \text{Op}_X, \text{Cov}_X)$ is weakly locally small and not small. If $X$ were locally small, then  $Y_0$ would be a subset of a small open set (Fact 2.1.21 in \cite{Pie2}), so $Y_0$ would be finite. Hence, $X$ is not locally small. We have $\{ Y_n: n\in \omega\}\in\text{EF}(\tau(\text{Op}_X), \{Y_n: n\in\omega\})$ and all the sets $Y_n$ are small and open in $(X, \text{EF}(\tau(\text{Op}_X),$ $\{Y_n: n\in\omega\}))$, so the gts $(X, \text{EF}(\tau(\text{Op}_X), \{Y_n: n\in\omega\}))$ is locally small. We put $\Psi=\text{Cov}_X\cup \text{EssFin}(\tau(\text{Op}_X))$. Then $pt(\text{Cov}_X)=\langle\Psi\rangle_X$ is the generalized topology of  $X_{pt}$. By Proposition 2.12, $\langle\Psi\rangle_X\subseteq  \text{EF}(\tau(\text{Op}_X), \{Y_n: n\in\omega\}) $.  Surprisingly, if $\mathbf{CC}$ holds, then  $X_{pt}$ is not locally small and, in consequence, $\langle \Psi\rangle_X\subset \text{EF}(\tau(\text{Op}_X), \{Y_n: n\in\omega\})$. To prove this, let us assume $\mathbf{ZF+CC}$. It is easy to observe the following facts:\\
\textbf{Fact 1.} $X\notin [X]^{\leq\omega}\cup_1 \mathbf{Sm}(X)$.\\
\textbf{Fact 2.}   Each $\Psi_n (n\in \omega)$ is closed with respect to restriction: $\Psi_n\cap_2 A\subseteq \Psi_n$ for $A\subseteq X$. 
(Notice that $\tau(\text{Op}_X)=\mathcal{P}(X)$  and $n=0$ is the hardest case.)\\
For $\mathcal{W}\subseteq\mathcal{P}(X)$, let us consider the following property: 

$\mathbf{P}(\mathcal{W})$: $\mathcal{W}$ has an uncountable member and $\mathcal{W}\subseteq [X]^{\leq\omega}\cup_1 \mathbf{Sm}(X)$.\\
For $n\in\omega$, let $T(n)$ be the statement:

$T(n)$:  if $\mathcal{W}\in\Psi_n$  has $\mathbf{P}(\mathcal{W})$, then $\mathcal{W}$ is essentially finite on $X\setminus A$ for some countable  $A\subseteq X$.\\
We are going to prove by induction that the following fact holds:\\
\textbf{Fact 3.} $T(n)$ is true for each $n\in \omega$.
\begin{proof}
Let $\mathcal{W}\in\Psi_0$ have property $\mathbf{P}(\mathcal{W})$.  Then, by Fact 1,  $X\notin\mathcal{W}$ and $\mathcal{W}\notin\text{Cov}_X$. Hence $\mathcal{W}\in \text{EssFin}(\tau(\text{Op}_X))$
and $T(0)$ holds. Suppose that $T(n)$ is true.  The \textit{finiteness}, \textit{stability} and \textit{regularity} induction steps from the proof of Proposition 2.2.37 in \cite{Pie1} are obvious (cf. Definition 2.2.2 of \cite{Pie1}).

\textit{Transitivity step.}  Let $\mathcal{W}\in\Psi_{n+1}$ have property $\mathbf{P}(\mathcal{W})$. Suppose that $\mathcal{U}\in{\Psi}_n$ and $\{ \mathcal{V}(U): U\in\mathcal{U}\}\subseteq\Psi_{n}$ are such that $\mathcal{W}=\bigcup_{U\in\mathcal{U}}\mathcal{V}(U)$ and, for each $U\in\mathcal{U}$, we have $U=\bigcup\mathcal{V}(U)$. Consider any $U\in\mathcal{U}$.  If every member of $\mathcal{V}(U)$ is countable, then $U\in[X]^{\leq\omega}$ because $\mathbf{CC}$ holds and $\mathcal{V}(U)$ is essentially countable. 
Suppose  $\mathcal{V}(U)$ has an uncountable member. Since $\mathcal{V}(U)$ has property $\mathbf{P}(\mathcal{V}(U))$, it follows from the inductive assumption that there is a countable set $A(U)\subseteq X$ such that $\mathcal{V}(U)$ is essentially finite on $X\setminus A(U)$. Then $U\in [X]^{\leq\omega}\cup_1 \mathbf{Sm}(X)$ and $U$ is uncountable. 
The above implies that $\mathcal{U}$ has property $\mathbf{P}(\mathcal{U})$. By the assumption, there is a countable $A\subseteq X$ such that $\mathcal{U}$ is essentially finite on $X\setminus A$. Let $\mathcal{U}^{\ast}\subseteq \mathcal{U}$ be a finite family such that $\bigcup\mathcal{U}^{\ast}\setminus A=\bigcup\mathcal{U}\setminus A$. For each $U\in \mathcal{U}^{\ast}$, the set $U$ is countable or $\mathcal{V}(U)$ is essentially finite on $U\setminus A(U)$. This implies that there is a countable $A(\mathcal{W})$ such that $\mathcal{W}$ is essentially finite on $X\setminus A(\mathcal{W})$. 

\textit{Saturation step.}  Suppose that there exists $\mathcal{V}\in\Psi_n$ such that $\bigcup\mathcal{V}=\bigcup\mathcal{W}$ and, for each $V\in\mathcal{V}$, there exists a non-empty $\mathcal{W}(V)=\{ W\in \mathcal{W}:V \subseteq W\}$. 
Since $\mathcal{W}\subseteq [X]^{\leq\omega}\cup_1\mathbf{Sm}(X)$, we have $\mathcal{V}\subseteq [X]^{\leq\omega}\cup_1 \mathbf{Sm}(X)$. Since $\mathcal{W}$ has an uncountable member and $\mathcal{V}$ is essentially countable, also $\mathcal{V}$ has an uncountable member and has property $\mathbf{P}(\mathcal{V})$. By the inductive assumption, there exists a countable $A(\mathcal{V})$ such that $\mathcal{V}$ is essentially finite on $X\setminus A(\mathcal{V})$. Then $\mathcal{W}$ is essentially finite on $X\setminus A(\mathcal{V})$, too. 
\end{proof}

Suppose that  $X_{pt}$ is locally small. There exists $\mathcal{W}\in pt(\text{Cov}_X)$ such that $\mathcal{W}\subseteq\mathbf{Sm}(X)$ and $X=\bigcup\mathcal{W}$. Since $X$ is uncountable and $\mathcal{W}$ is essentially countable, at least one member of $\mathcal{W}$ is uncountable, so $\mathbf{P}(\mathcal{W})$ holds true. By Fact 3, there exists a countable $A(\mathcal{W})$ such that $\mathcal{W}$ is essentially finite on $X\setminus A(\mathcal{W})$. Then $X\setminus A(\mathcal{W})\in\mathbf{Sm}(X)$. This is impossible by Fact 1.
\end{example}
 
From Fact 3.7, taken together with Example 3.8, we deduce the following corollary:
  
\begin{corollary} In every model for $\mathbf{ZF+CC}$, there exists a gts $X$ such that $X\neq X_{pt}$, while both $X$ and $X_{pt}$ are simultaneously weakly locally small and not locally small.
\end{corollary}

 We do not have a satisfactory solution to the following open problem:
 \begin{problem} Is it true in $\mathbf{ZF}$ that if the partial topologization of a gts $X$ is locally small, then so is $X$?
 \end{problem}

\begin{proposition} Suppose that  $X=(X, \text{Op}_X, \text{Cov}_X)$ is a gts and $\mathcal{B}$ is a bornology in $X$. Then the following conditions are equivalent:
\begin{enumerate}
\item[(i)] the gts $X$ is (quasi)-(pseudo)metrizable with respect to $\mathcal{B}$;
\item[(ii)] the gts $(X, \text{EF}(\tau(\text{Op}_X), \mathcal{B}))$ is $\mathbf{Sm}$-(quasi)-(pseudo)metrizable and the collection   
 $\tau(\text{Op}_X)\cap\mathcal{B}$ is a base for $\mathcal{B}$.
\end{enumerate}
\end{proposition}
\begin{proof} Assume that $(i)$ holds.  By Theorem 4.7 of \cite{PW2}, the collection $\tau(\text{Op}_X)\cap\mathcal{B}$ is a base for $\mathcal{B}$. It follows from Fact 2.16 that $\mathcal{B}$ is equal to the family $\mathbf{Sm}((X, \text{EF}(\tau(\text{Op}_X), \mathcal{B})))$. In consequence, $(i)$ implies $(ii)$. On the other hand,  we can use Fact 2.16 with  both Definitions 1.6 and 2.15 to infer that $(i)$ follows from $(ii)$.
\end{proof}

\begin{definition} Suppose that $(X, \mathcal{B})$ is a generalized bornological universe where $X=(X, \text{Op}_X, \text{Cov}_X)$.  Let us say that $X$ is \emph{strongly $\mathcal{B}$-(quasi)-(pseudo)\-metrizable} if there exists a (quasi)-(pseudo)metric $d$ on $X$ such that $\mathcal{B}$ is the collection of all $d$-bounded sets and $\text{Op}_X=L_X[\{B_d(x, r): x\in X \wedge  r\in (0; +\infty)\}]$.
\end{definition}

In connection with strong $\mathbf{Sm}$-(quasi)-(pseudo)metrizability, let us pose the following open problem:
 \begin{problem}
 Find useful simultaneously necessary and sufficient conditions for a gts to be strongly $\mathbf{Sm}$-(quasi)-(pseudo)metrizable. 
 \end{problem}

\begin{definition} A \emph{(quasi)-(pseudo)metric gts} is an ordered pair $(X, d)$ 
where $X=(X,\text{Op}_X, \text{Cov}_X)$ is a gts and $d$ is 
a (quasi)-(pseudo)metric on $X$ such that $\tau(d)=\tau(\text{Op}_X)$. 
\end{definition}

\begin{definition} For a (quasi)-(pseudo)metric gts $(X, d)$ and a bornology $\mathcal{B}$ in $X$, we say that $(X, d)$ is \emph{uniformly (quasi)-(pseudo)metrizable with respect to $\mathcal{B}$} or, equivalently, that $\mathcal{B}$ is \emph{uniformly (quasi)-(pseudo)metrizable with respect to $d$} if there exists a (quasi)-(pseudo)metric $\rho$ on $X$ such that $d$ and $\rho$ are uniformly equivalent, 
and $\mathcal{B}=\mathcal{B}(\rho,X)$ (cf. Definition 6.3 of \cite{PW2} and Definition 2.3 of \cite{GM}).
\end{definition}

The following proposition follows from Theorem 6.5 of \cite{PW2}:

\begin{proposition}
Suppose that $(X, d)$ is a (quasi)-(pseudo)metric gts. Then a bornology $\mathcal{B}$ in $X$ is uniformly (quasi)-(pseudo)metrizable with respect to $d$ if and only if
$\mathcal{B}$ has a base $\{ B_n: n\in\omega\}$ such that, for some $\delta\in (0; +\infty)$ and for each $n\in\omega$, the inclusion $[B_n]^{\delta}_d\subseteq B_{n+1}$ holds.
\end{proposition}

Other conditions that are equivalent to uniform  $\mathcal{B}$-(quasi)-(pseudo)metrizability of (quasi)-(pseudo)metric gtses can be deduced from the results of  \cite{GM} and from Section 6 of~\cite{PW2}.
 
\section{Applications to examples}

For $x,y\in\mathbb{R}$, let $d_n(x,y)=\mid x-y\mid$.  We denote by $\tau_{nat}$ the natural topology of $\mathbb{R}$ induced by the metric $d_n$. Let $\mathbf{CB}_{{nat}}(\mathbb{R})$ stand for the compact bornology of $(\mathbb{R}, \tau_{nat})$.  The topology $u=\{\emptyset, \mathbb{R}\}\cup\{ (-\infty; a): a\in\mathbb{R}\}$ is called the \emph{ upper topology} on $\mathbb{R}$, while  $l=\{\emptyset, \mathbb{R}\}\cup\{(a; +\infty): a\in\mathbb{R}\}$ is called the \emph{lower topology} on $\mathbb{R}$ (cf. \cite{FL}, \cite{Sal}). The following collections:
$$\mathbf{UB}(\mathbb{R})=\{ A\subseteq\mathbb{R}: \exists_{r\in\mathbb{R}} A\subseteq (-\infty; r)\},\:\: \mathbf{LB}(\mathbb{R})=\{ A\subseteq\mathbb{R}: \exists_{r\in\mathbb{R}} A\subseteq (r; +\infty)\}$$
are simple examples of bornologies in $\mathbb{R}$. Obviously,  
$$\mathbf{CB}_{nat}(\mathbb{R})=\mathbf{UB}(\mathbb{R})\cap \mathbf{LB}(\mathbb{R}).$$

\begin{example}
The topological space $(\mathbb{R}, u)$ is not quasi-metrizable (since it is not $T_1$) but it is quasi-pseudometrizable by $\rho_u(x,y)=\max(0,y-x)$.  
\begin{enumerate}
\item[(i)] For the gts $\mathbb{R}_{uu}=(\mathbb{R}, \text{EF}(u,\mathbf{UB}(\mathbb{R})))$, 
one has $\mathbf{ACB}(\mathbb{R}_{uu})\!=\mathbf{Sm}(\mathbb{R}_{uu})\!=\!\mathbf{UB}(\mathbb{R})$.
 This is why $\mathbb{R}_{uu}$ is  both $\mathbf{ACB}$- and $\mathbf{Sm}$-quasi-pseudometrizable by $\rho_u$. 
\item[(ii)] For the gts $\mathbb{R}_{ul}=(\mathbb{R}, \text{EF}(u,\mathbf{LB}(\mathbb{R})))$, we have $\mathbf{ACB}(\mathbb{R}_{ul})=\mathbf{Sm}(\mathbb{R}_{ul})=\mathcal{P}(\mathbb{R})$.
Hence $\mathbb{R}_{ul}$ is both $\mathbf{ACB}$- and $\mathbf{Sm}$-quasi-pseudometrizable by $\rho_{u,1}=\min\{1,\rho_u\}$.
\item[(iii)]  The gts $(\mathbb{R}, \text{EF}(u,\mathbf{CB}_{nat}(\mathbb{R})))$ is equal to $\mathbb{R}_{uu}$.  
\item[(iv)] The gts $(\mathbb{R}, \text{EF}(u,\mathcal{P}(\mathbb{R}))$ is equal to $\mathbb{R}_{ul}$.  
\item[(v)] The gts $\mathbb{R}_{uf}=(\mathbb{R}, \text{EF}(u, \mathbf{FB}(\mathbb{R})))$ is not $\mathbf{LB}(\mathbb{R})$-quasi-pseudometrizable because $\text{int}_uA=\emptyset$ for each $A\in\mathbf{LB}(\mathbb{R})$. Here $\mathbf{Sm}(\mathbb{R}_{uf})$ is the collection of all sets $A\in\mathbf{UB}(\mathbb{R})$ such that every non-empty subset of $A$ has its largest element. Similarly, $\mathbb{R}_{uf}$ is not $\mathbf{Sm}$-quasi-pseudometrizable.  Since $\mathbf{ACB}(\mathbb{R}_{uf})=\mathbf{CB}(\mathbb{R}_{uf})=\mathbf{UB}(\mathbb{R})$, the gts $\mathbb{R}_{uf}$ is $\mathbf{ACB}$-quasi-pseudometrizable by $\rho_u$. 
\item[(vi)] Each of $\mathbb{R}_{uu}, \mathbb{R}_{ul}, \mathbb{R}_{uf}$ is $\mathbf{CB}$-quasi-pseudometrizable by $\rho_u$.
\end{enumerate}
\end{example} 

Let us use the real lines described in Definition 1.2 of \cite{PW} as part of our illuminating examples for the notions of (uniform) $\mathcal{B}$-(quasi)-metrizability in the category $\mathbf{GTS}$.  

\begin{example}
For $x,y\in \mathbb{R}$, we put $ d_{n, 1}(x,y)=\min\{ d_n(x, y), 1\}$ and 
$$d_n^+(x,y)=d_n(\Phi(x),\Phi(y)) \text{ where } \Phi(x)= \left\{
\begin{array}{ll} e^x, & x<0,\\
1+x, & x\ge 0.\end{array} \right. $$
Moreover, we define  $d^+_{n,1}(x, y)=\min\{d^+_n(x, y), 1\}$. Let us observe that the metrics $d_n$ and $d^+_n$ are equivalent but not uniformly equivalent.
\begin{enumerate} 
\item[(i)]  Let $Cov$ be any generalized topology in $\mathbb{R}$ such that $((\mathbb{R}. \tau_{nat}, Cov), d_n^+)$ is a metric gts. 

We have $\mathcal{B}(d_n,\mathbb{R})=\mathbf{CB}_{{nat}}(\mathbb{R})$ and $\mathcal{B}(d^+_n,\mathbb{R})=\mathbf{UB}(\mathbb{R})$. Let us observe that, for a fixed $\delta\in (0; +\infty)$, there exists $n(\delta)\in\omega$ such that if $C_m=[-m; m]$ for $m\in\omega$ with $m>n(\delta)$,  then $(-\infty; m)\subseteq [C_m]^{\delta}_{d^+_n}$. This, together with Proposition 3.16, implies that $\mathcal{B}(d_n,\mathbb{R})$ is not uniformly quasi-metrizable with respect to $d^+_n$.  

\item[(ii)] For the usual topological real line $\mathbb{R}_{ut}$ (cf. Definition 1.2(i) of \cite{PW}), we have $\mathbf{FB}=\mathbf{Sm}\subset \mathbf{CB}=\mathbf{ACB}$ and $\text{int}_{{nat}}A=\emptyset$ for each $A\in\mathbf{Sm}(\mathbb{R}_{ut})$,  so the gts $\mathbb{R}_{ut}$ is not $\mathbf{Sm}$-quasi-metrizable and it is $\mathbf{ACB}$-metrizable by $d_n$. The metric gtses $(\mathbb{R}_{ut}, d_n)$ and $(\mathbb{R}_{ut}, d_{n,1})$ are $\mathbf{ACB}$-uniformly metrizable. It follows from (i) that the metric gtses $(\mathbb{R}_{ut}, d^+_n)$ is $(\mathbb{R}_{ut}, d^+_{n, 1})$ are not uniformly $\mathbf{ACB}$-quasi-metrizable. 

\item[(iii)] For the real lines $\mathbb{R}_{lst}$ and $\mathbb{R}_{lom}$ (cf. Definition 1.2(iv)-(v) of \cite{PW}), we have $pt(\mathbb{R}_{lom})=\mathbb{R}_{lst}$ and $\mathbf{Sm}=\mathbf{CB}=\mathbf{ACB}=\mathcal{B}(d_n,\mathbb{R})$. The metric gtses $(\mathbb{R}_{lst}, d_n)$ and $(\mathbb{R}_{lom}, d_n)$ are both uniformly $\mathbf{Sm}$-metrizable; however, none of the metric gtses $(\mathbb{R}_{lom}, d^+_n)$ and $(\mathbb{R}_{lst}, d^+_n)$ is uniformly $\mathbf{Sm}$-metrizable (see (i)).

\item[(iv)] For the real lines $\mathbb{R}_{l^+om}$ and $\mathbb{R}_{l^{+}st}$ (cf. Definition 1.2(vii)-(viii) of \cite{PW}), we have $pt(\mathbb{R}_{l^+om})=\mathbb{R}_{l^+st}$ and $\mathbf{CB}=\mathbf{CB}_{{nat}}(\mathbb{R})\subset \mathbf{Sm}=\mathbf{ACB}=\mathcal{B}(d^+_n,\mathbb{R})$. Now, it is obvious that both the metric gtses $(\mathbb{R}_{l^+om}, d^+_n)$ and $(\mathbb{R}_{l^+st}, d^+_n)$ are uniformly $\mathbf{ACB}$-metrizable by the metric $d^+_n$. The gtses $\mathbb{R}_{l^+om}$ and $\mathbb{R}_{l^{+}st}$ are $\mathbf{Sm}$-metrizable. The metric gtses $(\mathbb{R}_{l^+om}, d_n)$ and $(\mathbb{R}_{l^+st}, d_n)$ are uniformly $\mathbf{Sm}$-metrizable and uniformly $\mathbf{ACB}$-metrizable by 
$d_u(x,y)=d_{n,1}(x,y)+|\max(y,0)-\max(x,0)|$.

\item[(v)] Let us consider the gtses $\mathbb{R}_{om}, \mathbb{R}_{slom}, \mathbb{R}_{rom}$ and $\mathbb{R}_{st}$ (cf. Definition 1.2(ii), (iii), (vi) and (x) of \cite{PW}). We have $pt(\mathbb{R}_{om})= pt(\mathbb{R}_{slom})= pt(\mathbb{R}_{rom})=\mathbb{R}_{st}$ and $\mathbf{CB}\subset \mathbf{Sm}=\mathbf{ACB}=\mathcal{P}(\mathbb{R})$. The real lines $\mathbb{R}_{om}, \mathbb{R}_{slom}, \mathbb{R}_{rom}$ and $\mathbb{R}_{st}$  are $\mathbf{Sm}$-metrizable by the metric $d_{n, 1}$ and they are $\mathbf{CB}$-metrizable by the metric $d_{n}$.

\item[(vi)] The gts $\mathbb{R}_{om}$  (cf. Definition 1.2(ii) of \cite{PW}) is strongly $\mathbf{Sm}$-metrizable by $d_{n, 1}$.  
\end{enumerate}
\end{example}
 
A famous quasi-metrizable but non-metrizable Tychonoff space is the Sorgenfrey line, denoted here by $\mathbb{R}^S$. The topology $\tau_{S,r}$ of $\mathbb{R}^S$ is the the right half-open interval topology in $\mathbb{R}$. The space $\mathbb{R}^S$ is quasi-metrizable by the quasi-metric $\rho_{S}$ defined, for $x, y\in\mathbb{R}$,  as follows: 
 $$ \rho_{S}(x,y)=\left\{ \begin{array}{ll}
 y-x, & x\le y\\
 1, & x>y.\end{array}\right.$$
By Example 4.12 of \cite{PW2}, it is worthwhile to notice that the topology $\tau_{S,r}$ is also induced by the quasi-metric $\rho_{L}$ defined, for $x, y\in\mathbb{R}$, as follows:
 $$ \rho_{L}(x,y)=\left\{ \begin{array}{ll}
 \min\{y-x, 1\}, & x\le y\\
 1+x-y, & x>y.\end{array}\right.$$ 
 
Now, we may consider the following gtses, obtained from the Sorgenfrey line, that are modifications of the gtses defined in Definition 1.2 of \cite{PW}: 
 
\begin{definition} \label{proste Sorgenfreya}
We give names to the following Sorgenfrey real lines:
\begin{enumerate}
\item[(i)]  the \emph{usual topological Sorgenfrey real line} $\mathbb{R}^S_{ut}$
where $\text{Cov}=$ all families of members of $\tau_{S,r}$;

\item[(ii)] the \emph{o-minimal Sorgenfrey real line} $\mathbb{R}^S_{om}$
where $\text{Cov}=$ essentially finite families of finite unions of right half-open intervals;

\item[(iii)] the \emph{smallified topological} (or \emph{small partially topological}) \emph{Sorgenfrey real line} $\mathbb{R}^S_{st}$
where $\text{Cov}=$ essentially finite families of members of $\tau_{S,r}$;

\item[(iv)] the \emph{localized o-minimal Sorgenfrey real line} $\mathbb{R}^S_{lom}$
where $\text{Cov}=$ locally essentially finite families of locally finite unions of right half-open intervals; 

\item[(v)] the \emph{localized smallified topological Sorgenfrey real line} $\mathbb{R}^S_{lst}$
where $\text{Cov}=$ locally essentially finite families of members of $\tau_{S,r}$;

\item[(vi)] the \emph{smallified localized o-minimal Sorgenfrey real line} $\mathbb{R}^S_{slom}$ 
with $\text{Cov}=$ essentially finite families of locally finite unions of right half-open intervals;

\item[(vii)] the \emph{localized at $+\infty$} ($-\infty$, resp.) \emph{o-minimal Sorgenfrey real line} $\mathbb{R}^S_{l^+om}$ ($\mathbb{R}^S_{l^-om}$, resp.) where $\text{Cov}=$ locally essentially finite families of locally finite unions of right half-open intervals which, on the negative (positive, resp.) half-line,
are essentially finite and consist of only finite unions of right half-open intervals; 

\item[(viii)] the \emph{localized at $+\infty$} ($-\infty$, resp.) \emph{smallified topological Sorgenfrey real line} $\mathbb{R}^S_{l^+st}$ ($\mathbb{R}^S_{l^-st}$, resp.) where $\text{Cov}=$ locally essentially finite families of members of $\tau_{S,r}$ which are essentially finite on the negative (positive, resp.) half-line; 

\item [(ix)] the \emph{smallified localized at $+\infty$} ($-\infty$, resp.) \emph{o-minimal Sorgenfrey real line} $\mathbb{R}^S_{sl^+om}$ ($\mathbb{R}^S_{sl^-om}$, resp.) 
where $\text{Cov}=$ essentially finite families of locally finite unions of right half-open intervals which are only finite unions of right half-open intervals on the negative (positive, resp.) half-line;

\item[(x)] the \emph{rationalized o-minimal Sorgenfrey real line} $\mathbb{R}^S_{rom}$ where $\text{Cov}=$ essentially finite families of finite unions of right half-open intervals with endpoints being rational numbers or infinities. In this case, the topology $\tau(\bigcup\text{Cov})$ differs from $\tau_{S, r}$.
\end{enumerate}
\end{definition}

 \begin{example} (\textbf{Gtses from the Sorgenfrey line}.) 
 \begin{enumerate}
 \item[(i)] The gtses $\mathbb{R}^S_{lst}, \mathbb{R}^S_{lom}$   are both $\mathbf{ACB}$- and $\mathbf{Sm}$-quasi-metrizable by the quasi-metric $\rho_{0}$ defined as follows:
  $$ \rho_{0}(x,y)=\left\{ \begin{array}{ll}
 y-x, & x\le y\\
 1+x-y, & x>y.\end{array}\right.$$
\item[(ii)] The gtses $\mathbb{R}^S_{l^+st}, \mathbb{R}^S_{l^+om}$ are both $\mathbf{ACB}$- and $\mathbf{Sm}$-quasi-metrizable by $\rho_S$, while the gtses $\mathbb{R}^S_{l^-om}, \mathbb{R}^S_{l^-st} $ are both $\mathbf{ACB}$- and $\mathbf{Sm}$-quasi-metrizable by $\rho_L$.
 \item[(iii)] The gtses   $\mathbb{R}^S_{om}, \mathbb{R}^S_{slom}, \mathbb{R}^S_{st}, \mathbb{R}^S_{sl^+om}$  are $\mathbf{ACB}$- and $\mathbf{Sm}$-quasi-metrizable by $\rho_{S, 1}$.
 \item[(iv)] It follows from Theorem 1.9 that the gts  $\mathbb{R}^S_{ut}$ is neither $\mathbf{Sm}$-quasi-metrizable because $\tau_{S, r}\cap\mathbf{FB}(\mathbb{R})$ is not a base for $\mathbf{FB}(\mathbb{R})$, nor $\mathbf{ACB}$-quasi-metrizable (see below). 
 \end{enumerate}
 \end{example}
 
\begin{example}
Since each relatively compact set in the Sorgenfrey topology is countable, 
none of the Sorgenfrey lines defined above is $\mathbf{CB}$-metrizable.
\end{example}
 
\begin{example}(\textbf{Quasi-metric gtses from the Sorgenfrey line}.) We use the same notation as in Example 4.4.
\begin{enumerate}
\item[(i)] The quasi-metric gtses  $(\mathbb{R}^S_{lst}, \rho_{0})$, $(\mathbb{R}^S_{lom}, \rho_{0})$ are uniformly $\mathbf{ACB}=\mathbf{Sm}$-quasi-metrizable by $\rho_{0}$.
\item[(ii)] The quasi-metric gtses $(\mathbb{R}^S_{l^+st}, \rho_{0})$, $(\mathbb{R}^S_{l^+om}, \rho_{0})$ are uniformly $\mathbf{ACB}=\mathbf{Sm}$-quasi-metrizable by $\rho_{S}$, while the quasi-metric gtses  $(\mathbb{R}^S_{l^-st}, \rho_{0})$, $(\mathbb{R}^S_{l^-om}, \rho_{0})$ are uniformly $\mathbf{ACB}=\mathbf{Sm}$-quasi-metrizable by $\rho_{L}$,
\item[(iii)] The quasi-metric gtses $(\mathbb{R}^S_{om}, \rho_0)$, $(\mathbb{R}^S_{slom}, \rho_0)$,  $(\mathbb{R}^S_{st}, \rho_0)$, $(\mathbb{R}^S_{sl^+om}, \rho_0)$ are uniformly $\mathbf{ACB}=\mathbf{Sm}$-quasi-metrizable by $\min\{ \rho_0, 1\}$.
\item[(iv)] None of the quasi-metric gtses $(\mathbb{R}^S_{lom}, \rho^-_S)$,  $(\mathbb{R}^S_{lst}, \rho^-_S)$ is uniformly $\mathbf{ACB}$-
or $\mathbf{Sm}$-quasi-metrizable, where
$$ \rho^-_S (x,y)= \rho_S(\Phi(-y),\Phi(-x)) \qquad (\mbox{with }\Phi \mbox{ from Example 4.2})  .$$
\end{enumerate}
\end{example}

\begin{example} Let us put $J=[0; 1]\times\{0\}$ and $J_{q}=\{q\}\times [0; 1]$. For $S=[0; 1]\cap\mathbb{Q}$, let $X=J\cup\bigcup_{q\in S}J_q$. We consider the collection $\mathcal{B}$ of all sets $A\subseteq X$ that have the property: there exists a finite $S(A)\subseteq S$ such that $A\subseteq J\cup\bigcup_{q\in S(A)}J_q$.
\begin{enumerate}
\item[(i)] Let $e$ be the Euclidean metric on $X$. Then, for each $A\in\mathcal{B}$, we have $\text{int}_{\tau(e)}A=\emptyset$, so, for every topology $\tau_2$ in $X$,  the bornology $\mathcal{B}$ is not $(\tau(e), \tau_2)$-proper. In consequence, by Theorem 1.9, the gts $(X,\text{EF}(\tau(e), \mathcal{B}))$ is not $\mathbf{Sm}$-quasi-metrizable.
\item[(ii)] We define another metric $\rho$ on $X$ as follows. For $x, y\in [0; 1]$ and $q, q'\in S$ with $q\neq q'$, we put $\rho((x,0), (y,0))=\mid x-y\mid, \rho((q, x),(q, y))=\mid x-y\mid $ and $\rho((q, x), (q', y))=x+\mid q-q'\mid +y$. Then, for each $q\in S$ and for any $a, b\in [0; 1]$ with $a<b$, we have $\{ q\}\times (a; b)=\text{int}_{\tau(\rho)}[\{q\}\times (a; b)]\in\mathcal{B}$. Since there does not exist $A\in\mathcal{B}$ such that $J\subseteq\text{int}_{\tau(\rho)}A$, we deduce from Theorem 1.9  that the gts  $(X,\text{EF}(\tau(\rho), \mathcal{B}))$ is not $\mathbf{Sm}$-quasi-metrizable. The space $(X, \tau(\rho))$ can be called the \emph{comb with its hand $J$ and teeth $J_q$}, $q\in \mathbb{Q}$ (compare with Example IV.4.7 of \cite{Kn}).
\end{enumerate}
\end{example}
 
In what follows, if $X$ and $Y$ are gtses, the symbol $X\times_{\mathbf{GTS}} Y$ denotes their $\mathbf{GTS}$-product (cf. Definition 4.6 of \cite{PW}). 

\begin{example} 
The space $\mathbb{R}_{om}\times_{\mathbf{GTS}} \mathbb{R}_{lom}$ is $\mathbf{Sm}$-metrizable by 
$$d_{om,lom}((x_1,x_2),(y_1,y_2))=\big| \frac{x_1}{\sqrt{1+x_1^2}} - \frac{x_2}{\sqrt{1+x_2^2}}\big| + |x_2-y_2|,$$
 but $ \mathbb{R}_{ut}\times_{\mathbf{GTS}} \mathbb{R}_{lom}$ is not  $\mathbf{Sm}$-quasi-metrizable. Similarly, $\mathbb{R}_{st}\times_{\mathbf{GTS}} \mathbb{R}_{l^+om}$ is $\mathbf{Sm}$-metrizable. (See    Fact 4.10  in \cite{PW}.)
\end{example}

\section{New topological categories}

The table of categories in \cite{AHS}, among other categories, says about the category $\mathbf{Top}$ of topological spaces, the category $\mathbf{BiTop}$ of bitopological spaces and about the category $\mathbf{Bor}$ of bornological sets. The categories $\mathbf{GTS}$, $\mathbf{GTS}_{pt}$, $\mathbf{SS}$ of small generalized topological spaces and $\mathbf{LSS}$ of locally small generalized topological spaces, as well as  $\mathbf{SS}_{pt}$ and $\mathbf{LSS}_{pt}$, were introduced in \cite{Pie1} and \cite{Pie2}. 

In the light of  Proposition 2.14  and Fact 3.7(i), we can state the following:

\begin{f} The functor $pt$ of partial topologization preserves smallness and local smallness. More precisely:
\begin{enumerate}
\item[(i)] $pt$ restricted to $\mathbf{SS}$ maps $\mathbf{SS}$ onto $\mathbf{SS}_{pt}$;
\item[(ii)] $pt$ restricted to $\mathbf{LSS}$ maps $\mathbf{LSS}$ onto $\mathbf{LSS}_{pt}$. 
\end{enumerate}
\end{f}

The categories $\mathbf{Top}, \mathbf{BiTop},  \mathbf{GTS}, \mathbf{GTS}_{pt}, \mathbf{SS}, \mathbf{SS}_{pt}$ and $\mathbf{Bor}$ are all topological constructs (cf. \cite{AHS}, \cite{H-N}, \cite{Pie1},\cite{Pie2},  \cite{PW} and \cite{Sal}). Since $\mathbf{Top}$ and $\mathbf{Bor}$ are topological constructs, it is obvious that the category $\mathbf{UBor}$ of bornological universes (cf. Remark 2.2.70 of \cite{Pie1}) is a topological construct, too.  Let us define several more categories and answer the question whether they are topological constructs.

\begin{definition}[cf. 1.2.1 in \cite{H-N}] Let $\mathcal{B}_X$ be a boundedness  in a set $X$ and let $\mathcal{B}_Y$ be a boundedness  in a set $Y$. We say that a mapping $f:X\to Y$ is $(\mathcal{B}_X, \mathcal{B}_Y)$\emph{-bounded} (in abbreviation: \emph{bounded}) if, for each $A\in\mathcal{B}_X$, we have $f(A)\in\mathcal{B}_Y$. 
\end{definition}

\begin{definition} Suppose that $((X,\tau_1^X, \tau_2^X), \mathcal{B}_X)$ and $((Y,\tau_1^Y, \tau_2^Y), \mathcal{B}_Y)$  are bor\-no\-logical biuniverses. We say that a mapping $f:X\to Y$ is a \emph{bounded bicontinuous mapping} from $((X,\tau_1^X, \tau_2^X), \mathcal{B}_X)$ to $((Y,\tau_1^Y, \tau_2^Y), \mathcal{B}_Y)$ if $f$ is bicontinuous  with respect to $(\tau_1^X, \tau_2^X, \tau_1^Y, \tau_2^Y)$  and $f$ is $(\mathcal{B}_X, \mathcal{B}_Y)$-bounded. 
\end{definition}

\begin{definition} Suppose that $((X, \text{Cov}_X), \mathcal{B}_X)$ and $((Y, \text{Cov}_Y), \mathcal{B}_Y)$ are generalized bornological universes. We say that a mapping $f: X\to Y$ is a \emph{bounded strictly continuous mapping} from $((X, \text{Cov}_X), \mathcal{B}_X)$ to $((Y, \text{Cov}_Y), \mathcal{B}_Y)$ if $f$ is both $(\mathcal{B}_X, \mathcal{B}_Y)$-bounded and $(\text{Cov}_X, \text{Cov}_Y)$-strictly continuous. 
\end{definition}

\begin{definition} A generalized bornological universe $((X, \text{Cov}_X), \mathcal{B})$ is called: 
\begin{enumerate}
\item[(i)] \emph{partially topological} if the gts $(X, \text{Cov}_X)$ is partially topological;
\item[(ii)] \emph{small} if the gts $(X, \text{Cov}_X)$ is small.
\end{enumerate}
\end{definition}

\begin{definition} We define the following categories:
\begin{enumerate}
\item[(i)] $\mathbf{BiUBor}$ where objects are bornological biuniverses and morphisms are bounded bicontinuous mappings;
\item[(ii)] $\mathbf{GeUBor}$ where objects are generalized bornological universes and morphisms are bounded strictly continuous mappings;
\item[(iii)] $\mathbf{Ge_{pt}UBor}$ where objects are partially topological generalized bor\-no\-lo\-gi\-cal universes and morphisms are bounded strictly continuous mappings;
\item[(iv)] $\mathbf{SmUBor}$ where objects are small generalized  bornological universes and morphisms are bounded strictly continuous mappings;
\item[(v)] $\mathbf{Sm_{pt}UBor}$ where objects are partially topological small generalized bor\-no\-logical universes and morphisms are bounded strictly continuous mappings.
\end{enumerate}
\end{definition} 

 \begin{proposition}
 All categories $\mathbf{BiUBor}$, $\mathbf{GeUBor}$, $\mathbf{Ge_{pt}UBor}$, $\mathbf{SmUBor}$ and
 $\mathbf{Sm_{pt}UBor}$  are topological constructs.
 \end{proposition}
 \begin{proof} To check that, for instance,  $\mathbf{Ge_{pt}UBor}$ is a topological construct, we mimic the proof to Theorem 4.4 of \cite{PW}. Namely, let us consider a source $F=\{ f_i: i\in I\}$ of mappings $f_i: X\to Y_i$ indexed by a class $I$ where every $Y_i$ is a partially topological generalized bornological universe and $Y_i=((X_i, \text{Cov}_i), \mathcal{B}_i)$. Let $\text{Cov}_X$ be the $\mathbf{GTS}$-initial generalized topology for $F$ in $X$ (cf. Definition 4.3 of \cite{PW}) and let 
 $\mathcal{B}_X=\bigcap_{i\in I}\{ A\subseteq X: f_i(A)\in\mathcal{B}_i\}$.  
 For $X=((X, \text{Cov}_X), \mathcal{B}_X)$, let $X_{pt}=(pt((X, \text{Cov}_X)), \mathcal{B}_X)$. The canonical morphism  $id: X_{pt}\to X$ is such that all mappings $f_i\circ id$ are morphisms in $\mathbf{Ge_{pt}UBor}$. For any object $Z$ of $\mathbf{Ge_{pt}UBor}$ and a mapping $h:Z\to X_{pt}$, we can observe that if all $f_i\circ id\circ h$ with $i\in I$ are morphisms, then $id\circ h$ is a morphism of $\mathbf{GTS}$, so $pt(h)=h$ is a morphism of $\mathbf{GTS}_{pt}$. If all $f_i\circ id\circ h$ are bounded, then $pt(h)=h$ is bounded, too. That $\mathbf{BiUBor}, \mathbf{GeUBor}, \mathbf{SmUBor}$ and $\mathbf{Sm_{pt}UBor}$ are topological can be proved by using similar arguments. 
 \end{proof} 
 
Some other topological constructs, relevant to bornologies or  quasi-pseudo\-metrics, were considered in \cite{CL} and \cite{Vr}.

\textbf{Acknowledgement.} We thank Prof. W. Paw\l ucki for turning our attention to a mistake in an earlier version of Example 3.8, which has made it possible for us to correct the example.

\end{document}